\documentclass[12pt]{amsart}
\usepackage{amssymb,latexsym}

\usepackage{amsfonts}
\usepackage{amsthm}
\usepackage{amsmath}
\usepackage{amscd}
\usepackage{enumerate}
\newtheorem*{definition*}{Definition}
\newtheorem*{question*}{Question}

\theoremstyle{remark}

\numberwithin{equation}{section}

\theoremstyle{plain}
\newtheorem{lemma}{Lemma}[section]
\newtheorem{theorem}[lemma]{Theorem}

\newtheorem{corollary}[lemma]{Corollary}

\theoremstyle{definition}

\newtheorem{question}[lemma]{Question}

\numberwithin{equation}{section}
\def\begeq{\stepcounter{lemma}\begin{equation}}

\newcommand{\lra}{\longrightarrow}

\newcommand{\nn}{\mathbb{N}}

\newcommand{\F}{\mathbb{F}}
\newcommand{\Q}{\mathbb{Q}}

\newcommand{\Z}{\mathbb{Z}}

\DeclareMathOperator{\res}{res} 
\DeclareMathOperator{\Gal}{Gal} \DeclareMathOperator{\gr}{gr}
 
\DeclareMathOperator{\Ker}{Ker}

\begin{document}

\title{The Bloch-Kato conjecture and galois theory}
\author[D. Karagueuzian]{Dikran Karagueuzian}
\address{SUNY at Binghamton \\ Department of Mathematical
Sciences \\ P. O. Box 6000 \\ Binghamton, NY 13902-6000}
\email{dikran@math.binghamton.edu}
\author[J. Labute]{John Labute}
\address{McGill University \\ Department of Mathematics and
Statistics \\ 805 Sherbrooke Street West \\ Montreal \\ Quebec \\
Canada H3A 2K6}
\email{labute@math.mcgill.ca}
\author[J. Min\'{a}\v{c}]{J\'{a}n Min\'{a}\v{c}}
\address{The University of Western Ontario \\ Department of
Mathematics \\ Middlesex College \\ London \\ Ontario \\ Canada N6A 5B7}
\email{minac@uwo.ca}
\thanks{The second and third authors are partially supported by
NSERC grants.}

\date{October 11, 2009}

\begin{abstract}
We investigate the relations in Galois groups of maximal
$p$-extensions of fields, the structure of their natural
filtrations, and their relationship with the Bloch-Kato
conjecture proved by Rost and Voevodsky with Weibel's
patch. Our main focus is on the third degree, but we
provide examples for all degrees.
\end{abstract}

\dedicatory{To Paulo Ribenboim who showed us the Galois road}

\maketitle

\parskip=10pt plus 2pt minus 2pt

\section{Introduction}\label{S0}

Let $p$ be a fixed prime number and let $F$ be a field such that
$F$ contains a primitive $p$-th root of unity $\zeta_p$. Let $F_s$
be the separable closure of $F$ and $\Gal(F_s/F)$ be the absolute
Galois group of $F$. Let $H^*(F) = H^*(\Gal(F_s/F),
\mathbb{Z}/p\mathbb{Z})$ be the Galois cohomology of $F$ in
coefficients $\mathbb{Z}/p\mathbb{Z}$. (From now on we shall omit
the coefficients of our cohomology groups as they will always be
$\mathbb{Z}/p\mathbb{Z}$.) In \cite{mi}, Milnor defined the group
$K_nF$ by generators and relations for any $n\in \nn \cup \{0\}$.
The generators are $n$-tuples $\{a_1,\dots ,a_n\}$ of elements of
$F^*$, and the defining relations are the multiplicativity in each
component and the Steinberg relation $\{a_1,\dots ,a_n\} = 0$ if
$a_i + a_j = 1$ for some $i\not= j$. Set $k_nF = K_nF/pK_nF$ and
$k_*F = \bigoplus_{n\geq 0} k_nF$. Thus $k_*F$ is the Milnor
K-theory of $F$ modulo $p$. In the same paper Milnor defined a
graded homomorphism $h_*:k_*F\to H^*(F)$ and implicitly
conjectured that it is an isomorphism when $p=2$. Because of the
later important work of Bloch and Kato for any prime $p$, the
general case is now known as the Bloch-Kato conjecture.

Let $F(p)$ be the maximal $p$-extension of $F$.  This means that
$F(p)$ is a union of Galois extensions $K/F$, such that
$\Gal(K/F)$ is a $p$-group, in a fixed $F_s$ of $F$.  Let $G$ be
the Galois group of $F(p)/F$.  Observe that $\inf:H^1(G)\to
H^1(F)$ is an isomorphism and the Steinberg relations hold in
$H^2(G)$.  Then $h_*:k_*F\to H^*(F)$ factors through $h_*:k_*F\to
H^*(G)$. A simple application of the Hochschild-Serre spectral
sequence shows that if $h_n:k_nF\to H^n(F)$ is isomorphic for all
fields $F$ then $h_n:k_nF\to H^n(G)$ is an isomorphism for all
fields $F$. Also see \cite[page 97, for the case $p=2$]{g-m}. In
the paper \cite{m-s}, Merkurjev and Suslin proved that $h_2$ is an
isomorphism. Recently Rost and Voevodsky, with Weibel's patch,
proved that $h_n$ is an isomorphism for all $n\in \nn$. (See
\cite{ha-w}, \cite{ro1}, \cite{ro2}, \cite{voe1}, \cite{voe2},
\cite{voe3}, and \cite{wei1}, \cite{wei2}.) The cohomology of
Galois groups of maximal pro-$p$-extensions of fields reflects
some properties that the class of these groups share. The
relationship between the structure of groups and their cohomology
groups is in general nontrivial and it is quite mysterious. We are
interested in the group theoretic meaning of the Bloch-Kato
conjecture.

We investigate some strong, and at the same time simple conditions
on the relations of $\Gal(F(p)/F)$ for any field $F$ which is
closely related to the Bloch-Kato conjecture. These conditions say
that all relations of $G=\Gal(F(p)/F)$ are generated by relations
of small weight. This will be made more precise in section 4. For
related work on the structure of $\Gal(F(p)/F)$ and its relations
with cohomology, see \cite{a-k-m}, \cite{j-w}, \cite{ko1},
\cite{ko2}, \cite{mi-sp1} and \cite{n-s-w}.

Let $\{\sigma_i\}_{i\in I}$ be a minimal set of generators of $G$. (See
\cite[Chapter~4]{ko1}.) We assume that $I$ is well ordered.  Let $S$ be a free
pro-$p$-group with a minimal set of generators $\{s_i\}_{i\in I}$.  By
sending $s_i$ to $\sigma_i$ we obtain a continuous homomorphism $\pi :S\to
G$.  We set $R = Ker(\pi)$.

For a pro-$p$-group $G$, we define $G^{(n+1)} =
(G^{(n)})^p[G^{(n)},G]$ and $G^{(1)} = G$.  Then $G^{(n)}$ is a
closed normal subgroup of $G$. We denote the quotient by
$G^{[n]}$.  We set $R^{(1,S)} = R$ and $R^{(n+1,S)} =
(R^{(n,S)})^p[R^{(n,S)},S]$ for $n\geq 1$.  Since $S$ and $G$ have
the same cardinality of the minimal set of generators, then
$R\subset S^{(2)}$.  In general, we see by induction on $n$ that
$R^{(n,S)}\subset R\cap S^{(n+1)}$. Lemma~\ref{L2.2} says that
if $h_2$ is surjective, then $R^{(2,S)} = R\cap S^{(3)}$.

{\bf Example A.} The equality $R^{(2,S)}=R \cap S^{(3)}$ implies
that every $\sigma \in G=\Gal(F(p)/F)$ of finite order $\ne1$ has
order $p$ as follows. Suppose that $\sigma$ has order $\ge p^2$.
The subgroup of $G$ generated by $\sigma$ is a closed subgroup of
$G$ and if $L$ is its fixed field, we have $L(p) = F(p)$ and
$\Gal(L(p)/L) = \langle\sigma\rangle$. Hence we may assume that $F
= L$ and $G = \langle\sigma\rangle$. Moreover by taking a suitable
power of $\sigma$ we may assume that $\sigma$ has an order $p^2$.
Then $\sigma^{p^2}\in R \cap S^{(3)}$ but $\sigma^{p^2} \notin
R^{(2,S)}$ as $\sigma^{p^2}$ generates $R$ as a normal subgroup of
$S$. Hence the order of $\sigma$ is $p$. This fact also follows
from the work of Becker~\cite{be} where he also shows that $p=2$
and that two elements of $G$ order $2$ cannot commute. In fact
the only non-trivial finite subgroup of $G$ is the cyclic group
of order $2$.

{\bf Example B.} (See \cite[Section~9]{ch-e-m} for more examples.)
One can also deduce from the equality  $R^{(2,S)}=R \cap S^{(3)}$ that a minimal se of
relations among the $\sigma_i$ cannot contain a relation of the
form $[[\sigma_1, \sigma_2],\sigma_3]$ with $\sigma_1,
\sigma_2,\sigma_3$ distinct. Indeed, such a relator would be in
$R\cap S^{(3)}$ but not in $R^p[R,S]=R^{(2,S)}$.

Our first result is

\begin{theorem}\label{T1}
Assume that $G$ is $G_F(p)$ for some field $F$ containing a
primitive $p$th-root of $1$. Then $R^{(3,S)} = R\cap S^{(4)}$.
\end{theorem}

Theorem \ref{T1} is proved in section \ref{S2}.  It is natural
to ask the following question.

\begin{question}\label{Q2}
{\it{Let $G$ be isomorphic to a Galois group of a maximal p-extension of
a field. Under which condition is $R^{(n,S)} = R\cap S^{(n+1)}$ for all
$n\geq 1$?}}
\end{question}

More generally we ask:
\medskip

\begin{question}\label{Q3}
{\it{Let $G$ be isomorphic to the Galois group of the maximal
$p$-extension of a field. Describe the quotients $R\cap
S^{(n+1)}/R^{(n,S)}$ for all $n\geq 1$. }}
\end{question}

In the first two sections we show that $R^{(n,S)} = R\cap
S^{(n+1)}$ for $n\leq 3$. In section \ref{S3}, we give an
equivalent description of Question \ref{Q2} in the language of Lie
algebras.  We prove that, for quadratically defined pro-$2$-groups
as well as for $G=\Gal(F(2)/F)$ when $F$ is a totally imaginary
number field, the relation $R^{(n,S)}=R\cap S^{(n+1)}$ is valid
for each $n\geq 1$. For all odd primes $p$ we show the same holds
for $G=\Gal(F(p)/F)$ where $F$ is any local or global field.

\section{Preliminaries}\label{S1}

We use the following usual notation:  $[a]$ means
both an element of $F^*/F^{*p}$ and its corresponding element
$(a)$ in $H^1(F)$ or more generally in $H^1(\Gal(T/F))$ where
$T/F$ is any Galois extension of $F$ which contains $F^{(2)}$:  =
compositum of all cyclic extensions of degree $p$ of $F$, and is a subfield
of $F(p)$. Observe that by Kummer
theory we have the canonical isomorphism $F^*/F^{*p}\to
H^1(\Gal(T/F))$, which justifies our identification mentioned
above.

We will work in the category of pro-$p$-groups and make the usual
conventions that by ``subgroup'' we mean ``closed subgroup'', by
``generated'' we mean ``topologically generated'' and by ``morphism''
we mean ``continuous morphism''.  We shall always work with $\mod p$
cohomology except when we explicitly mention other coefficients.
For the standard facts for Galois cohomology we refer to \cite{se2}
and \cite{n-s-w}.

For an extension $1\to A\to B\to C\to 1$ of profinite groups and a
discrete $B$-module $M$ we have the corresponding Lyndon-Hochschild-Serre
(LHS, for short) spectral sequences $\{E^{p,q}_r,d^{p,q}_r\}$ where
$E^{p,q}_2 = H^p(C,H^q(A,M))$ (\cite{h-s} and \cite{n-s-w}).
Using LHS we also have the five-term exact sequence
\begin{align}
& 0\to H^1(C,M^A) \overset{\inf}{\longrightarrow} H^1(B,M)\overset
{\res}{\longrightarrow} H^1(A,M)^C \overset{d^{0,1}_2}{\longrightarrow}
H^2(C,M^A) \notag \\
& \overset{\inf}{\longrightarrow} H^2(B,M) \notag
\end{align}
where $\inf$ is the inflation map and $\res$ is the restriction
map.

Observe that $k_nF$ is isomorphic to the factor group of
\[ F^*/F^{*p}\otimes\cdots\otimes F^*/F^{*p} \]
by the subgroup generated by
$[a_1]\otimes\cdots\otimes [a_n]$ where $a_1$ is a norm in the
extension $F(\root{p}\of{a_j})/F$ of degree $p$, for
some $j = 2,\dots,n$. (If $F(\root{p}\of{a_j})/F$ is not an extension
of degree $p$, then $[a_j]=[1]$ and $[a_1]\otimes \dots \otimes
[a_n]=0$.)

Indeed let $l_nF$ be this factor group.  Because all relations in
$l_nF$ are valid in $k_nF$, we see that $k_nF$ is a naturally
homomorphic image of $l_nF$.  (See \cite[page~303, Exercise~6]{f-v}.)
In order to show that this homomorphism
is actually an isomorphism, it is enough to show that the Steinberg
relations which generate the relations in $k_nF$ are valid in $l_nF$
also.  Let $a_1,\dots,a_n\in F^*$. We denote as $\langle a_1,\dots,a_n
\rangle$ the image of $[a_1]\otimes\dots\otimes [a_n], [a_i]\in
F^*/F^{*p}$ in $l_nF$.

We want to show that if $1=a_i + a_j, 1\le i<j\le n$, then
\[
\langle a_1,\dots,a_n\rangle = 0 \mbox{ in }l_nF.
\]
Because this is trivially true if $a_j\in F^{*p}$ we shall
assume that $a_j\notin F^{*p}$. If $i=1$ then this is true
by our definition of the relations in $l_nF$ as
\begin{align*}
1-a_j &= \prod^{p-1}_{i=0}(1-\zeta_p^i \root{p}\of{a_j})\\
      &= N_{F(\root{p}\of{a_j})/F}(1-\root{p}\of{a_j}).
\end{align*}

Hence it is enough to show that if
\[
1 < i < j \le n
\]
then
\[
\langle a_1,\dots,a_i,\dots,a_j,\dots,a_n \rangle =
-\langle a_i,\dots,a_1,\dots,a_j,\dots,a_n \rangle.
\]
However using the equation
\[
-a = (1-a)/(1-a^{-1}) \mbox{ for }a \ne 1,
\]
we see that
\[
\langle a,\dots,-a,\dots \rangle = 0 \mbox{ in }l_nF
\mbox{ for all }a\in F^*.
\]
Hence
\begin{align*}
  0 &= \langle a_1 a_i,\dots,-a_1 a_i,\dots a_n \rangle \\
    &= \langle a_1,\dots,-a_1,\dots \rangle + \langle
    a_1,\dots,a_i,\dots \rangle + \\ & \langle a_i,\dots,
    a_1,\dots \rangle + \langle a_i,\dots,-a_i,\dots \rangle \\
    &= \langle a_1,\dots,a_i,\dots \rangle + \langle a_i,\dots,
    a_1,\dots \rangle,
\end{align*}
as required.

Hence $l_nF = k_nF$.

\section{The proof of Theorem~\ref{T1}}\label{S2}

We divide our proof into several lemmas.  Lemma~\ref{L2.3}
shows that it is enough to prove that $\inf:H^2(S^{[3]})
\to H^2(S/R^{(2,S)})$ is surjective.  This is proved by
considering certain spectral sequences below.

\begin{lemma}\label{L2.1}
$R^{(n,S)}\subset R\cap S^{(n+1)}$.
\end{lemma}

\begin{proof}
Our inclusion is true if $n = 1$.  Assume that it is true for $k\leq n$.
Then $R^{(n+1,S)} = (R^{(n,S)})^p[R^{(n,S)},S]\subset
(S^{(n+1)})^p[S^{(n+1)},S] = S^{(n+2)}$.  So $R^{(n+1,S)}\subset R\cap
S^{(n+2)}$.
\end{proof}

Lemma \ref{L2.2} below was observed in \cite[page 102]{wur} under
an additional hypothesis, and in
\cite[page 57]{mi-sp2} for the case $p=2$. This lemma was also
generalized in \cite[page 207]{g-m}
in the case $p=2$ and in \cite{ch-e-m} for all $p$.

\begin{lemma}\label{L2.2}
If $h_2$ is surjective then $R^p[R,S] = R^{(2,S)} = R\cap
S^{(3)}$.
\end{lemma}

We consider the following pair of extensions
\begeq\label{CD2.3}\begin{CD}
1 @>>> R^{(2,S)} @>>> S @>>> S/R^{(2,S)} @>>> 1 \\
@. @VVV @VVV @VVV \\
1 @>>> S^{(3)} @>>> S @>>> S^{[3]} @>>> 1\; .
\end{CD}
\end{equation}

By applying the five-term exact sequence to \eqref{CD2.3}, we obtain a
commutative diagram
\begeq\label{CD2.4}\begin{CD}
H^1(R^{(2,S)})^{S/R^{(2,S)}} @>{\cong}>> H^2(S/R^{(2,S)}) \\
@AA\res A @AA\inf A \\
H^1(S^{(3)})^{S^{[3]}} @>{\cong}>> H^2(S^{[3]})\; .
\end{CD}
\end{equation}

The surjectivities of both isomorphisms follow from the fact that
$H^2(S) = 0$, because $S$ is a free pro-$p$-group.  The injectivities
follow from the fact that $H^1(S/R^{(2,S)})\cong H^1(S)$ and
$H^1(S^{[3]})\cong H^1(S)$, since both $R^{(2,S)}$ and $S^{(3)}$ are
normal subgroups of $S^{(2)}$.

Observe that $H^1(S^{(3)}/S^{(4)})\cong H^1(S^{(3)})^{S^{[3]}}$
and $H^1(R^{(2,S)}/R^{(3,S)})\cong H^1(R^{(2,S)})^{S/R^{(2,S)}}$.
The restriction $H^1(S^{(3)}/S^{(4)})\to H^1(R^{(2,S)}/R^{(3,S)})$
is given by the composite
\begin{align}
&H^1(S^{(3)}/S^{(4)})\overset{\res}{\longrightarrow}H^1(R^{(2,S)}S^{(4)}/S^{(4)})\overset{\cong}{\longrightarrow}H^1(R^{(2,S)}/R^{(2,S)}\cap S^{(4)}) \notag \\
&\overset{\inf}{\longrightarrow}H^1(R^{(2,S)}/R^{(3,S)})\; .
\notag
\end{align}
The restriction $H^1(S^{(3)}/S^{(4)})\to
H^1(R^{(2,S)}S^{(4)}/S^{(4)})$ is surjective.  If $R^{(2,S)} =
R\cap S^{(3)}$, then $R^{(2,S)}\cap S^{(4)} = (R\cap S^{(4)}) =
R\cap S^{(4)}$. Then $R^{(3,S)} = R\cap S^{(4)}$ iff $R^{(3,S)} =
R^{(2,S)}\cap S^{(4)}$. Also the last inflation map is surjective
iff $R^{(2,S)} \cap S^{(4)} = R^{(3,S)}$. Thus the last map $\inf:
H^1(R^{(2,S)}/R^{(2,S)}\cap S^{(4)} \to H^1(R^{(2,S)}/R^{(3,S)})$
is surjective iff $R^{(3,S)} = R \cap S^{(4)}$.
Therefore we obtain

\begin{lemma}\label{L2.3}
Assume that $R^{(2,S)} = R\cap S^{(3)}$.  The following are equivalent
    \begin{enumerate}
        \item $R^{(3,S)} = R\cap S^{(4)}$.
        \item $R^{(3,S)} = R^{(2,S)}\cap S^{(4)}$.
        \item $\inf:H^1(R^{(2,S)}/R^{(2,S)}\cap S^{(4)})\to H^1(R^{(2,S)}/R^{(3,S)})$ is surjective.
        \item $\res:H^1(S^{(3)}/S^{(4)})\to H^1(R^{(2,S)}/R^{(3,S)})$ is surjective.
        \item $\inf:H^2(S^{[3]})\to H^2(S/R^{(2,S)})$ is surjective.
    \end{enumerate}
\end{lemma}

We have a pair of extensions
\begeq\label{CD2.6}\begin{CD}
1 @>>> R/R^{(2,S)} @>>> S/R^{(2,S)} @>>> S/R = G @>>> 1 \\
@. @VVV @VVV @VVV \\
1 @>>> S^{(2)}/S^{(3)} @>>> S^{[3]} @>>> S^{[2]} @>>> 1\; .
\end{CD}
\end{equation}
To prove (5) in Lemma~\ref{L2.3} we can compare the LHS spectral
sequences corresponding to \eqref{CD2.6}.  We set
$E^{p,q}_r(S^{[3]})$ (or $E^{p,q}_r(S/R^{(2,S)})$) the $E^{p,q}_r
-$term corresponding to the bottom extension (or the top
extension).

Observe that $H^1(R/R^{(2,S)})\cong H^1(R)^G$ and $H^1(R)^G\cong
H^2(G)$, the last isomorphism follows from the five-term exact
sequence corresponding to the extension $1\to R\to S\to G\to 1$.
Then $H^1(R/R^{(2,S)})^G = H^1(R/R^{(2,S)})\cong H^2(G)$.
Similarly we have $H^1(S^{(2)}/S^{(3)})\cong H^2(S^{[2]})$.
Therefore the transgression maps $d_2^{0,1}$ in both spectral
sequences are surjective. Thus we obtain

\begin{lemma}\label{L2.7}
$E^{2,0}_\infty(S^{[3]}) = E^{2,0}_\infty(S/R^{(2,S)}) = 0$.
\end{lemma}

To prove the surjectivity of $\inf:H^2(S^{[3]})\to
H^2(S/R^{(2,S)})$, it is enough to show the surjectivities of the
homomorphism corresponding to $E^{1,1}_\infty$-terms and the
homomorphism corresponding to $E^{0,2}_\infty$-terms.

Assume that $R^{(2,S)} = R\cap S^{(3)}$.  Observe that
\[ H^1(R/R^{(2,S)}) = H^1(R/R\cap S^{(3)})\cong H^1(RS^{(3)}/S^{(3)})\;
. \]
Thus we obtain

\begin{lemma}\label{L2.8}
$\res\colon H^1(S^{(2)}/S^{(3)})\to H^1(R/R^{(2,S)})$ is surjective.
\end{lemma}
Since both extensions in \ref{CD2.6} are central extensions, then
$E^{0,2}_2(S^{[3]})\cong H^2(S^{(2)}/S^{(3)})$ and
$E^{0,2}_2(S/R^{(2,S)})\cong H^2(R/R^{(2,S)})$.  Let $\{z_c\}_{c\in C}$
be a basis for $H^1(R/R^{(2,S)})$ and $\{w_d\}_{d\in D}$ be a basis for
$H^1(S^{(2)}/S^{(3)})$. Let $\beta w_d$ and $\beta z_c$ be Bocksteins
of $w_d$ and $z_c$ in $H^2(R/R^{(2,S)})=E_3^{0,2}(S/R^{(2,S)})$ and
$H^2(S^{(2)}/S^{(3)})=E_3^{0,2}(S^{[3]})$ respectively. Then
$\{\beta w_d\}_{d\in D}$ and $\{\beta z_c\}_{c\in C}$ generate
$E_3^{0,2}(S/R^{(2,S)})$ and $E_3^{0,2}(S^{[3]})$ respectively.
This follows from a standard argument exploiting the fact that
$d_2$ is a derivation
with respect to the multiplicative structure of $E_2^{p,q}$ and
the fact that $d_2^{0,1}$ is injective.  However using the
fact that $d_2^{1,1}$ is surjective, we
see that $d_3^{0,2}(S^{[3]})=0$.  Hence $E_3^{0,2}(S^{[3]})=
E_\infty^{0,2}(S^{[3]})$.  In Lemma~3.9 we denote the natural map
$E^{0,2}_3(S^{[3]})\to E^{0,2}_3(S/R^{(2,S)})$ as $\res$ because
it is induced by the restriction map
$S^{(2)}/S^{(3)}\lra R/R^{(2,S)}$.
By Lemma \ref{L2.8}, and our discussion above, we obtain

\begin{lemma}\label{L2.9}
$\res\colon E^{0,2}_3(S^{[3]})\to E^{0,2}_3(S/R^{(2,S)})$ is surjective
and $E_3^{0,2}(S^{[3]})=E_\infty^{0,2}(S^{[3]})$.
\end{lemma}

We have a commutative diagram

\[ \begin{CD}
E^{0,2}_\infty (S/R^{(2,S)}) @>>> E^{0,2}_3(S/R^{(2,S)}) \\
@AAA @AAA \\
E^{0,2}_\infty (S^{[3]}) @= E^{0,2}_3(S^{[3]})\; .
\end{CD} \]

Using Lemma~\ref{L2.9} we deduce the following corollaries.

\begin{corollary}\label{C2.11}
$\res:E^{0,2}_\infty(S^{[3]})\to E^{0,2}_\infty(S/R^{(2,S)})$ is
surjective.
\end{corollary}

\begin{corollary}\label{C2.12}
$E^{0,2}_\infty(S/R^{(2,S)}) = E^{0,2}_3(S/R^{(2,S)})$.
\end{corollary}

Since both extensions in the above are central, then
\[
d^{1,1}_2:H^1(G,H^1(R/R^{(2,S)}))\to H^3(G)
\]
is the composite $H^1(G,H^1(R/R^{(2,S)}))
\overset{(1,d^{0,1}_2)}{\underset{\cong}{\longrightarrow}}
H^1(G)\otimes H^2(G)\overset{\cup}{\longrightarrow} H^3(G)$.
Also $d^{1,1}_2:H^1(S^{[2]}, H^1(S^{(2)}/S^{(3)}))\to H^3(S^{[3]})$
is the composite
\[ H^1(S^{[2]}, H^1(S^{(2)}/S^{(3)}))
\overset{(1,d^{0,1}_2)}{\underset{\cong}{\longrightarrow}}
H^1(S^{[2]})\otimes H^2(S^{[2]})\overset{\cup}{\longrightarrow}
H^3(S^{[2]})\; . \] We may abuse notation by setting
$$
E^{1,1}_\infty (S/R^{(2,S)}) = \Ker(\cup : H^1(G)\otimes
H^2(G)\to H^3(G))
$$
and $E^{1,1}_\infty (S^{[3]}) = \Ker(\cup :H^2(S^{[2]})\otimes
H^2(S^{[2]})\to H^3(S^{[2]}))$. We have a commutative diagram
\begeq\label{CD2.13}\begin{CD}
0 @>>> E^{1,1}_\infty (S/R^{(2,S)}) @>>> H^1(G)\otimes H^2(G) @>\cup >> H^3(G) @>>> 0 \\
@. @AA\text{\it u}A @AA\text{\it v}A @AA\text{\it inf}A \\
0 @>>> E^{1,1}_\infty (S^{[3]}) @>>> H^1(S^{[2]})\otimes H^2(S^{[2]}) @>\cup >>
H^3(S^{[2]}) @>>> 0.
\end{CD}
\end{equation}
Here $u$ and $v$ are natural maps induced by inflation maps.
Recall that $S^{[2]}\cong G^{[2]}$ under our projection map
$S \to G$. In the next lemma we use both the injectivity of
$h_3$ and the surjectivity of $h_2$.

\begin{lemma}\label{L2.14}
The map $u\colon E^{1,1}_{\infty}(S^{[3]})\lra E^{1,1}_{\infty}(S/R^{(2,S)})$ is surjective.

\end{lemma}

\begin{proof} We consider $E^{1,1}_{\infty}(S/R^{(2,S)})$ as
a subgroup of $H^1(G)\otimes H^2(G)$. Using the surjectivity
of $h_2$, we see that each element in $H^1(G)\otimes H^2(G)$
can be written as a sum of elements of the
form $\alpha\otimes(\beta\cup\gamma)$ where $\alpha,\beta,\gamma
\in H^1(G)$. As we remarked at the end of Section 2, the injectivity
of $h_3$ implies that $E^{1,1}_{\infty}(S/R^{(2,S)})$ is generated
by the elements $z_1\otimes (z_2\cup z_3)$ such that $z_1\cup z_2=0$
in $H^2(G)$. (Since $h_2$ is the isomorphism, we can work in
$H^1(G)$ rather than in $k_2F$.) Because the map $\inf\colon H^1
(S^{[2]})\lra H^1(G)$ is an isomorphism, we see that for the element
$z_1\otimes(z_2\cup z_3)$ as above we can find $y_i\in H^1(S^{[2]}),
i=1,2,3$ such that $\inf(y_i)=z_i$. Using our assumption that
$z_1\cup z_2=0$ in $H^2(G)$ we see that

$$u((y_1\otimes(y_2\cup y_3)-y_3\otimes(y_1\cup y_2))=z_1\otimes(z_2\cup z_3).$$

Therefore we see that our map $u$ is surjective.
\end{proof}

\section{Graded Lie Algebras}\label{S3}

Here we give an equivalent description of Question~1.2 in Lie
algebra language.  A convenient reference is Lazard's paper
\cite{laz1}. As usual, we consider a minimal presentation of $G$.

\begeq\label{CD3.2}\begin{CD}
1\longrightarrow R\longrightarrow S\longrightarrow
G\overset{\pi}\longrightarrow 1\; .
\end{CD}
\end{equation}

Let $S$ and $G$ admit the usual filtrations

\[ S^{(1)} = S,\dots ,S^{(n+1)} = (S^{(n)})^p[S^{(n)},S],\dots \]
\[ G^{(1)} = G,\dots ,G^{(n+1)} = (G^{(n)})^p[G^{(n)},G],\dots\; . \]

The formulae $G^{(n+1)}\subset G^{(n)},[G^{(n)},G^{(m)}]\subset
G^{(n+m)}$ imply that $\gr_n(G) = G^{(n)}/G^{(n+1)}$ (denoted
additively) is an vector space over $\F_p$ and that the graded
algebra $\gr(G) = \sum \gr_n(G)$ is an algebra over $\F_p$ where
multiplication of homogenous elements of $\gr(G)$ is induced by
the commutator operation.  This operation satisfies the Jacobi
identity and hence is a Lie bracket. The $p$-th power map in $G$
induces an operator $P$ on $\gr(G)$ making it a Lie algebra over
$\F_p[\pi]$ if $p\ne2$ and a mixed Lie algebra if $p=2$ (cf.
\cite{laz1}). Similarly, the induced filtration $\{R\cap
S^{(n)}\}$ yields the graded Lie algebra $\gr(R)_{ind}$. The
extension \ref{CD3.2} induces an exact sequence of (mixed) Lie
algebras
\[ 0\to \gr(R)_{ind}\to \gr(S)\overset{\phi}{\to} \gr(G)\lra 0. \]
Since the filtration of $G$ is discrete the map $\phi$ is
surjective. (See pages 428-430 in \cite{laz1} for details.)

On the other hand, the filtration $\{R^{(n,S)}\}$ yields another
Lie algebra $\gr(R,S)$.  The inclusion $R^{(n,S)}\subset R\cap
S^{(n+1)}$ (see Lemma \ref{L2.1} induces the homomorphism
$\iota_n:\gr_n(R,S)\to \gr_{n+1}(R)_{ind}$ and therefore a Lie
algebra homomorphism $\iota:\gr(R,S)\to \gr(R)_{ind}$. Let $U$ be
the enveloping algebra of $\gr(S)$. Then $U$ is the free
associative $\F_p[\pi]$-algebra on the free generators of $S$.
There is a canonical embedding of $\gr(S)$ into $U$ with $P(x)=\pi
x$ if $p\ne2$. If $p=2$ we have $P(x)=x^2+\pi x$ if $x$ is of
degree $1$ and $P(x)=\pi x$ if $x$ is of degree $>1$. Note that
$\gr(R,S)$ and $\gr(R)_{ind}$ are $U$-modules via the adjoint
representation and that $\iota$ is a homomorphism of $U$-modules.

\begin{theorem}\label{equiv}
The following are equivalent
\begin{enumerate}[{\rm(A)}]
\item We have $R^{(n,S)} = R\cap S^{(n+1)}$ for $n\ge1$.
\item The homomorphism $\iota$ is injective.
\item The homomorphism $\iota$ is surjective.
\end{enumerate}
\end{theorem}
\begin{proof}
Note that (A) holds for $n=1$ since $R\subset S^{(2)}$.

Assume that $\iota$ is injective and that (A) holds for some
$n\ge1$. Let $x\in R\cap S^{(n+2)}$. Then $x\in R\cap
S^{(n+1)}=R^{(n,S)}$. If $\xi$ is the image of $x$ in $\gr_n(R,S)$
we have $\iota_n(\xi)=0$ which implies $x\in R^{(n+1,S)}$. Hence
(A) holds for $n+1$ and by induction for all $n$.

Assume that $\iota$ is surjective and let $x\in R\cap S^{(n+1)}$.
Then there exists $y_0\in R^{(n,S)}$ such that $x_1=y_0^{-1}x\in
R\cap S^{(n+2)}$. In the same way we define inductively $y_i$ such
that $y_i\in R^{(n+i,S)}$ and $x_{i+1}=y_i^{-1}x_i\in R\cap
S^{(n+2+i)}$ for $i\ge0$ with $x_0=x$. Then $x=\prod y_i\in
R^{(n,S)}$.
\end{proof}

\begin{corollary}
Let $G=S/R$ be a minimal presentation of $G$ of finite type. Then
$i$ is surjective if $\iota(\gr_1(R,S))$ generates $\gr(R)_{ind}$
as an ideal of $\gr(S)$.
\end{corollary}

The elements of $i(\gr_1(R,S))\subset \gr_2(S)$ are images of
elements $r$ of $R$ under the canonical mapping of $S^{(2)}$ onto
$\gr_2(S)$. These images are called initial forms of the elements
$r$. If $\iota$ is bijective and $R\ne 1$ then $R$ has a minimal
generating set whose initial forms are of degree $2$. In this case
the presentation $G=S/R$ is called {\bf quadratic}. If $G=F/R$ is
of finite type we say it is {\bf quadratically defined} if it is
quadratic and the set of initial forms of a minimal generating set
for $R$ generate $\gr(R)_{ind}$ as an ideal of $\gr(S)$. The
latter is true if the set of initial forms is strongly free, cf.
\cite{lm}, \cite{lab2}. The group $G$ is said to be quadratically
defined if it has a minimal presentation which is quadratically
defined.

\begin{theorem}
Let $G=S/R$ be a minimal presentation of $G$ of finite type. Then
$i$ is bijective iff $R=1$ or $G=S/R$ is quadratically defined.
\end{theorem}

\begin{proof}
If $G=S/R$ is quadratically defined then $i(\gr_1(R,S))$ is a
generating set for $\gr(R)_{ind}$ as an ideal of $\gr(S)$. Then
$i$ is surjective since it is a $U$-module homomorphism and hence
bijective by  Theorem~\ref{equiv}.

Conversely, suppose that $i$ is bijective and identify $\gr(R,S)$
with its image in $\gr(S)$. To prove that $G=S/R$ is quadratically
defined it suffices to prove that $\gr_1(R,S)$ generates
$\gr(R,S)$ as a $U$-module. Let $M$ be the $U$-submodule of
$\gr(R,S)$ generated by $\gr_1(R,S)$. We have $M_1=\gr_1(R,S)$.
Suppose that $M_n=\gr_n(R,S)$ and let $\xi\in\gr_{n+1}(R,S)$,
$\xi\ne0$. If $x\in R^{(n+1,S)}$ is a representative of $\xi$ then
$x$ is a product of elements of the form $u^p$, $[u,v]$ with $u\in
R^{(n,S)}$, $v\in S$. Since $i$ is injective these elements lie in
$\gr_{n+1}(S)$ unless the degree of $u$ is $n$ and the degree of
$v$ is $1$. It follows that $\xi$ is a linear combination of
elements of the form $\pi\eta$, $[\eta,\zeta]$ with $\eta\in M_n$,
$\zeta\in U_1$ and hence that $\xi\in M_{n+1}.$
\end{proof}

\begin{question}
Let $G$ be isomorphic to the Galois group of a maximal
$p$-extension of a field, and let $i$ be a natural graded Lie
algebra homomorphism $i:\gr(R,S)\to \gr(R)_{ind}$. When is $i$ an
isomorphism?
\end{question}

If $F$ is a global field of characteristic $\ne p$ which is
totally imaginary if $F$ is a number field and $p=2$, then by
the results of \cite{lm} and \cite{sch}, $G=\Gal(F(p)/F)$ is a
projective limit of quadratically defined presentations. More
precisely, the group $G$ has a presentation $S/R$ where $S=\cup
S_i$ {$i\ge 1$} with $S_i\subset S_{i+1}$ finitely generated and,
if $R_i$ is the image of $R$ under the canonical projection of $S$
onto $S_i$, we have $S_i/R_i$ quadratically defined for all $i$.
By Theorem~\ref {equiv} this means that $R_i^{(n,S_i)}=R_i\cap
S_i^{(n+1)}$ for all $i,n$ which implies $R^{(n,S)}=R\cap
S^{(n+1)}$ for all $n$ since $S$ is the projective limit of the
$S_i$. Hence, by Theorem~\ref{equiv}, the map $\iota$ is an
isomorphism. The same is true if $F$ is a local field with
$\zeta_p$ in $F$ since then $G=\Gal(F(p)/F)$ is a Demushkin group
which is quadratically defined by \cite{lm}, \cite{lab2}. We thus
obtain the following result.

\begin{theorem} Let $F$ be a field containing a primitive $p$-th root of unity. If
$F$ is a global field, which is totally imaginary if $F$ is a
number field and $p=2$, or a local field containing a primitive
$p$-th root of unity then $F$ is quadratically defined.
\end{theorem}

\begin{question}
Is $\Gal(\Q(2)/\Q)$ quadratically defined?
\end{question}

\section{Acknowledgements}\label{S4}

We are very grateful to Alejandro Adem and Wenfeng Gao, who
made important contributions towards this work in its early
stages.  Alejandro Adem's continuous interest, and discussions,
have been a considerable encouragement for our work.

\end{document}